\documentclass[letterpaper,10pt]{amsart}
\usepackage{indentfirst} 
\usepackage{amssymb}
\usepackage{amsmath}
\usepackage{mathabx} 
\usepackage{amsthm}
\usepackage{thmtools}
\usepackage{bbm}             
\usepackage[linktocpage=true]{hyperref}  
\usepackage{tikz}
\usepackage{enumerate}
\usepackage{enumitem}
\usepackage{dsfont}
\usepackage[all]{xy} 
\usepackage[numbers]{natbib}

\declaretheorem{theorem}

\declaretheorem{lemma}

\declaretheorem{proposition}

\declaretheorem{remark}
\declaretheorem{claim}

\declaretheorem[name=Theorem,numbered=no]{theorem*}

\newcommand{\R}{\mathbb{R}}
\newcommand{\Z}{\mathbb{Z}}
\newcommand{\N}{\mathbb{N}}

\DeclareMathOperator{\Span}{span}

\DeclareMathOperator{\SL}{SL}
\DeclareMathOperator{\GL}{GL}

\DeclareMathOperator{\Aut}{Aut}

\DeclareMathOperator{\rank}{rank}
\DeclareMathOperator{\Ad}{Ad}

\def\phi{\varphi}
\def\R{{\mathbb R}}

\def\N{{\mathbb N}}
\def\Z{{\mathbb Z}}

\def\L{{\Lambda}}

\def\Q{{\mathbb Q}}

\def\1{\mathds{1}}

\def\rest#1#2{\mathchoice
              {\setbox1\hbox{${\displaystyle #1}_{\scriptstyle #2}$}
              \restrictionaux{#1}{#2}}
              {\setbox1\hbox{${\textstyle #1}_{\scriptstyle #2}$}
              \restrictionaux{#1}{#2}}
              {\setbox1\hbox{${\scriptstyle #1}_{\scriptscriptstyle #2}$}
              \restrictionaux{#1}{#2}}
              {\setbox1\hbox{${\scriptscriptstyle #1}_{\scriptscriptstyle #2}$}
              \restrictionaux{#1}{#2}}}
\def\restrictionaux#1#2{{#1\,\smash{\vrule height .8\ht1 depth .85\dp1}}_{\,#2}}

\title{Infinite Approximate Subgroups of Soluble Lie Groups}
\author{Simon Machado \\
University of Cambridge}
\date{\today}

\begin{document}

\begin{abstract}
We study infinite approximate subgroups of soluble Lie groups. Generalising a theorem of Fried and Goldman we show that approximate subgroups are close, in a sense to be defined, to genuine connected subgroups. Building up on this result we prove a structure theorem for approximate lattices in soluble Lie groups. This extends to soluble Lie groups a theorem about quasi-crystals due to Yves Meyer.
\end{abstract}

\maketitle

\section{Introduction}
Approximate subgroups were defined by Terence Tao in \cite{MR2501249} in order to give a non-commutative generalisation of results from additive combinatorics. On the one hand, finite approximate subgroups have been extensively studied in particular by Ehud Hrushovski \cite{MR2833482} and by Emmanuel Breuillard, Ben Green and Terence Tao \cite{MR3090256}, leading to the structure theorem \cite{MR3090256}. This asserts that finite approximate subgroups are commensurable to coset nilprogressions, which are a certain non-commutative generalisation of arithmetic progressions. On the other hand, it seems hopeless to aim at classifying all infinite approximate subgroups. Some results in this direction for particular classes of infinite approximate subgroups can be found in \cite{MR2833482}, \cite{MR3438951} and \cite{MR3345797}. Inspired by Yves Meyer's results on quasi-crystals (\cite{meyer1972algebraic}),  Michael Bj\"orklund and Tobias Hartnick have defined a class of infinite approximate subgroups called approximate lattices in \cite{bjorklund2016approximate}. These approximate subgroups generalise lattices (discrete subgroups of Lie groups with finite co-volume) and share many properties with them. For instance, lattices and approximate lattices in nilpotent Lie groups have a very similar theory, see \cite{machado2018approximate}. Whether similar results hold for other types of locally compact groups is the open question that motivates this article. Here, we address the case of soluble Lie groups (Theorem \ref{Meyer's theorem soluble} below). Along the way, we show a structure theorem for all approximate subgroups in soluble algebraic groups (Theorem \ref{Main theorem}).  
\medbreak

A subset $\L$ of a group $G$ containing the identity is an \emph{approximate subgroup} if it is symmetric, i.e. $\L=\L^{-1}$, and if there exists a finite subset $F \subset G$ such that $\L^2 \subset F\L$. Here, $\L^2:=\{\lambda_1\lambda_2 | \lambda_1,\lambda_2 \in \L\}$, $F\L:=\{f\lambda |f \in F, \lambda \in \L \}$ and more generally $\L^n:=\{\lambda_1\cdots \lambda_n | \lambda_1,\ldots,\lambda_n \in \L \}$. Moreover, set $\L^{\infty} := \bigcup_{n \geq 0} \L^n$ the subgroup generated by $\L$. We will say that two subsets $\L, \Xi \subset G$ are \emph{commensurable} if there is a finite set $F$ such that $\L \subset F\Xi$ and $\Xi \subset F\L$. If $G$ is endowed with the structure of a topological group, we say that subsets $\L,\Xi \subset G$ are \emph{compactly commensurable} if there is a compact subset $K \subset G$ with $\L \subset K\Xi$ and $\Xi \subset K\L$. Commensurability and compact commensurability are equivalence relations.
 
 An approximate subgroup $\L \subset G$ in a locally compact group is a \emph{uniform approximate lattice} if it is discrete and compactly commensurable to $G$. The approximate group condition arises naturally from the combination of discreteness and compact commensurability to the ambient group: if a subset $\L \subset G$ is symmetric, compactly commensurable to $G$ and $\L^6$ is discrete, then $\L$ is a uniform approximate lattice. See \cite{bjorklund2016approximate} for this and more on the general theory of approximate lattices.
 
 Examples of uniform approximate lattices are given by \emph{cut-and-project schemes}. A cut-and-project scheme $(G,H,\Gamma)$ is the datum of two locally compact groups $G$ and $H$, and a uniform lattice $\Gamma$ in $G \times H$ such that $\Gamma\cap G = \{e\}$ and $\Gamma$ projects densely into $H$. Given a cut-and-project scheme $(G,H,\Gamma)$ and a symmetric relatively compact neighbourhood $W_0$ of $e_H$ in $H$, one gets a uniform approximate lattice when considering the projection $\L$ of $(G\times W_0 )\cap \Gamma$ to $G$. Any approximate subgroup of $G$ which is commensurable to such a $\L$ is called a \emph{Meyer subset} of $G$. This construction was first introduced by Yves Meyer in the abelian case \cite{meyer1972algebraic} and extended by Michael Bj\"orklund and Tobias Hartnick \cite{bjorklund2016approximate}. 
 
 In a similar fashion, for $\L\subset G$ a symmetric subset, we say that a group homomorphism $f: \L^{\infty}\rightarrow H$ with $H$ a locally compact group is a \emph{good model (for $\L$)} if:  $(i)$ $f(\L)$ is relatively compact, and $(ii)$ there is $V$ a neighbourhood of the identity in $H$ such that $f^{-1}(V)\subset \L$. In this situation, we say that $\L$ \emph{has a good model}.  In particular, note that if $\L$ has a good model, then $\L$ is an approximate subgroup.
 
 If $\L:=(G \times W_0) \cap \Gamma $ is a uniform approximate lattice constructed from a cut-and-project scheme $(G,H,\Gamma)$, then $f=p_H \circ (\rest{p_G}{\Gamma})^{-1} $ is a good model for $\L$, where $p_G$ and $p_H$ are the natural projections on $G$ and $H$ respectively. Conversely, if $\L\subset G$ is a uniform approximate lattice and has a good model $f$, the map 
 \begin{align*}
 \L^{\infty} & \rightarrow G\times \overline{f(G)} \\
 \gamma & \mapsto (\gamma, f(\gamma))
 \end{align*}
 embeds $\L^{\infty}$ in $G\times \overline{f(G)}$ as a uniform lattice. Thus, $(G,\overline{f(G)},\L^{\infty})$ is a cut-and-project scheme. Therefore, both constructions are equivalent and we will use the latter as it is handier in our case. For further results on good models in groups see \cite{machado2019goodmodels}.
 
 Now, we state our main results. The first theorem is concerned with general approximate subgroups in soluble algebraic groups. 
 
 \begin{theorem}\label{Main theorem}
  Let $ \L \subset \GL_n(\R)$ be an approximate subgroup generating a soluble subgroup. Then $\L$ is compactly commensurable to a Zariski-closed soluble subgroup of $\GL_n(\R)$ that is normalised by and contained in the Zariski-closure of an approximate subgroup commensurable to $\L$. 
 \end{theorem}
 
 Theorem \ref{Main theorem} is a non-commutative generalisation of a theorem due to Jean-Pierre Schreiber \cite[Proposition 2]{schreiber1973approximations}, which was recently given a new proof by Alexander Fish in \cite[Theorem 2.2]{fish2019extensions}. 
 Theorem \ref{Main theorem} also generalises a result of Fried and Goldman about the existence of \emph{syndetic hulls} for virtually solvable subgroups of $\GL_n(\R)$ (see \cite[Theorem 1.6]{MR689763} and Proposition \ref{proposition Fired et Goldman} below). Another interesting corollary to this result is that \emph{strong approximate lattices} (see \cite[Definition $4.9$]{bjorklund2016approximate}) in soluble algebraic groups are uniform (see Theorem \ref{strong approximate lattices in soluble groups}).

 \medbreak
 
 In \cite{meyer1972algebraic} Yves Meyer proved a structure theorem for what later came to be known as mathematical quasi-crystals. Quasi-crystals correspond to uniform approximate lattices in locally compact abelian groups. Rephrased with our terminology, Meyer's theorem becomes :
 
 \begin{theorem*}[Theorem 3.2,\cite{meyer1972algebraic}]
  Let $\L$ be a uniform approximate lattice in a locally compact abelian group $G$. Then $\L$ is a Meyer subset. 
 \end{theorem*}
 
 Motivated by this result the authors of \cite{bjorklund2016approximate} asked whether similar results would hold for other classes of locally compact groups \cite[Problem $1.$]{bjorklund2016approximate}. We answer this question in the soluble Lie case. This improves, using completely different methods, a previous article by the author that dealt with uniform approximate lattices in nilpotent Lie groups \cite{machado2018approximate}. 
 
%
%

 \begin{theorem}\label{Meyer's theorem soluble}
 Let $\Lambda \subset G$ be a uniform approximate lattice in a connected soluble Lie group. Then $\Lambda$ is a Meyer subset.
 \end{theorem}
 
 Let us now give a brief overview of the proof strategy for Theorems \ref{Main theorem} and \ref{Meyer's theorem soluble}. 
 
  Theorem \ref{Main theorem} will be proved by induction on the derived length. We use induction to reduce the proof to the case where $\left\{\lambda_1\lambda_2\lambda_1^{-1}\lambda_2^{-1} | \lambda_1,\lambda_2 \in \L \right\}$ is relatively compact. Then we are able to show that $\L$ is close to the centre of $\mathbb{G}\left(\R\right)$. The crux of the proof relies on the following fact that is specific to algebraic group homomorphisms: if $\phi$ is a algebraic group homomorphism and $S$ is a set that has relatively compact image by $\phi$, then $S$ is contained in $\ker(\phi)K$ for some compact subset $K$. Applied to inner automorphisms, this yields a result reminiscent of a classical theorem of Schur, according to which a group with a finite set of commutators has a finite-index centre. Finally, we conclude using ideas developed by Alexander Fish in his new proof of the abelian case \cite{fish2019extensions}.
  
 In order to prove Theorem \ref{Meyer's theorem soluble}, we first show that, although $\L^{\infty}$ is \emph{a priori} only a soluble group, $\L$ is commensurable to a uniform approximate lattice $\L'$ that generates a polycyclic group. Using Auslander's embedding theorem on polycyclic groups we embed $\left(\L'\right)^{\infty}$ as a lattice in some soluble algebraic group. Then $\L'$ is a Meyer subset according to Theorem \ref{Main theorem}. 
 \section{Technical Results about Commensurability }
 
 In this section, we prove two technical results that will turn out to be particularly helpful. These are well-known results in the theory of finite approximate subgroups, but their proofs do not use the finiteness assumption.  
 
 The first result is about the intersection of commensurable approximate subgroups. 
  
 \begin{lemma}
 \label{Intersection commensurable approximate subgroups}
 Let $\L,\Xi \subset G$ be approximate subgroups such that $\L \subset F\Xi$ for some finite set $F$. Then $\L^2 \cap \Xi^2$ is commensurable to $\L$. Moreover, there is $F' \subset \L$ with $|F'|=|F|$ such that $A \subset F'(\L^2 \cap \Xi^2)$. 
 \end{lemma}

\begin{proof}
 Let $F$ be a finite set such that $\L \subset F\Xi$ and for all $f \in F$ pick $x_f \in \L \cap f\Xi$. Then for any $x \in \L \cap f\Xi$ we have $x_f^{-1}x \in \L^2 \cap \Xi^2$ so $\L \subset \bigcup x_f \L^2 \cap \Xi^2 \subset \L^3$. 
\end{proof}

In particular, we can see that when $\L$ and $\Xi$ are genuine subgroups, commensurability as defined here is equivalent to commensurability of subgroups.

In the same line of ideas, we have the following lemma about intersections of general approximate subgroups. 

\begin{lemma}\label{Intersection of approximate subgroups}
 Let $\L,\Xi \subset G$ be approximate subgroups. Then $\L^2 \cap \Xi^2$ is an approximate subgroup. Moreover, $(\L^k \cap \Xi^k)_{k \geq 2}$ is a family of pairwise commensurable approximate subgroups. 
\end{lemma}

\begin{proof}
Let $F_1,F_2 \subset \Gamma$ be finite subsets such that $\Lambda^2 \subset F_1\Lambda$ and $\Xi^2 \subset F_2\Xi$. Define 
$$ T:=\{ (t_1,t_2) \in F_1^3\times F_2^3|t_1\Lambda \cap t_2\Xi \neq \emptyset \}.$$

For all $t=(t_1,t_2) \in T$ pick $x_t \in t_1\Lambda \cap t_2\Xi$. Then for all $x \in t_1\Lambda\cap t_2\Xi$ we have $x_t^{-1}x \in \Lambda^2 \cap \Xi^2$, hence $x \in x_t(\Lambda^2 \cap \Xi^2)$. So,
$$ t_1\Lambda \cap t_2\Xi \subset x_t(\Lambda^2 \cap \Xi^2).$$
Therefore, 
\begin{align*}
 (\Lambda^2\cap \Xi^2)^2 \subset \Lambda^4\cap \Xi^4 & \subset F_1^3\Lambda\cap F_2^3\Xi \\
						 & \subset \bigcup\limits_{t\in T}t_1\Lambda \cap t_2\Xi \\
						 & \subset \bigcup\limits_{t\in T}x_t(\Lambda^2\cap \Xi^2).
\end{align*}
So $\Lambda^2\cap \Xi^2$ is an approximate subgroup. Moreover, the previous inclusions show that 
$$\Lambda^4\cap \Xi^4 \subset  \bigcup\limits_{t\in T}x_t(\Lambda^2\cap \Xi^2). $$
So $\Lambda^4\cap \Xi^4$ and $\Lambda^2\cap \Xi^2$ are commensurable and by induction $\Lambda^{2^n}\cap \Xi^{2^n}$ is commensurable to $\Lambda^2\cap \Xi^2$.
 
\end{proof}

In particular, we will often use Lemma \ref{Intersection of approximate subgroups} with $\Xi$ a subgroup. Then $(\L^k \cap \Xi)_{k \geq 2}$ is a family of pairwise commensurable approximate subgroups. 

\section{Approximate subgroups in soluble linear groups}

In this section we prove Theorem \ref{Main theorem}. 
Let us first consider approximate subgroups in vector spaces. 

\begin{theorem}[Proposition 2,\cite{schreiber1973approximations}]\label{Approximate subgroup in vector spaces}
 Let $V$ be a real vector space and $\Lambda\subset V$ an approximate subgroup. There exists a vector subspace $W\subset V$ compactly commensurable to $\L$. 
\end{theorem}
 
 In our proof of Proposition \ref{Approximate subgroups in soluble algebraic Groups} we will need a slightly different result that is an easy consequence of Theorem \ref{Approximate subgroup in vector spaces}. 
 
 \begin{lemma}\label{lemma approximate subgroups in vector spaces}
 Let $V$ be a real vector space and $\Lambda\subset V$ a symmetric subset such that $\L+\L$ is compactly commensurable to $\L$. There exists a vector subspace $W\subset V$ compactly commensurable to $\L$.
\end{lemma}
 
 \begin{proof}
  According to Theorem \ref{Approximate subgroup in vector spaces} we only need to show that $\L$ is compactly commensurable to an approximate subgroup. 
  
  Let $U \subset V$ be a symmetric compact neighbourhood of $0$, then $\L + U$ is a symmetric set compactly commensurable to $\L$. Moreover, let $K \subset V$ be a compact subset such that $\L + \L \subset \L + K$ and $F \subset V$ be a finite subset such that $K + U + U  \subset U + F$. Then we have, 
  
  $$ (\L + U) + ( \L + U) \subset \L + K + U + U \subset (\L + U) + F.$$
 \end{proof}

Now, we will extend Theorem \ref{Approximate subgroup in vector spaces} to soluble real algebraic groups. In the proof of the following proposition we rely on the theory of algebraic groups. See \cite{springer2010linear} for a general introduction to linear algebraic groups.

\begin{proposition}\label{Approximate subgroups in soluble algebraic Groups}
 Let $\Lambda\subset \mathbb{G}(\R)$ be an approximate subgroup in the group of $\R$-points of a Zariski-connected soluble real algebraic group such that $\Lambda^{\infty}$ is Zariski-dense. Then there is $H \triangleleft \mathbb{G}(\R)$ closed connected normal subgroup such that $\Lambda$ is compactly commensurable to $H$. Moreover, $H$ is the connected component of the identity of the group of $\R$-points of some normal algebraic subgroup of $\mathbb{G}$.
\end{proposition}

\begin{proof}[Proof of Proposition \ref{Approximate subgroups in soluble algebraic Groups}]

 As $\Lambda^{\infty}$ is Zariski-dense, we know that $\Lambda^{\infty} \cap [\mathbb{G}(\R),\mathbb{G}(\R)]$ is Zariski-dense in $[\mathbb{G}(\R),\mathbb{G}(\R)]$. Moreover, $[\mathbb{G}(\R),\mathbb{G}(\R)]$ is a connected simply connected nilpotent Lie group so $\overline{\L^{\infty} \cap [\mathbb{G}(\R),\mathbb{G}(\R)]}$ is co-compact by \cite[Theorem $2.1$]{raghunathan1972discrete}. As a consequence, there is $k \in \N$ greater than $2$ such that $\Lambda':=\Lambda^k \cap [\mathbb{G}(\R),\mathbb{G}(\R)]$ is an approximate subgroup with $\Lambda'^{\infty}$ Zariski-dense in $[\mathbb{G}(\R),\mathbb{G}(\R)]$. 
 
 According to the induction hypothesis there is a closed connected subgroup $H_1 \triangleleft [\mathbb{G}(\R),\mathbb{G}(\R)]$  compactly commensurable to $\Lambda'$. In addition, for all $\lambda \in \Lambda$, we have $\lambda\left( \Lambda'\right) \lambda^{-1} \subset \Lambda^{k+2} \cap [\mathbb{G}(\R),\mathbb{G}(\R)]$. But, according to Lemma \ref{Intersection of approximate subgroups} approximate subgroups $\Lambda^{k+2}\cap [\mathbb{G}(\R),\mathbb{G}(\R)]$ and $\Lambda'$ are commensurable. Therefore, $H_1$ and $\lambda H_1 \lambda^{-1}$ are compactly commensurable.
 
 \begin{claim}\label{Claim normality}
 The subgroups $H_1$ and $\lambda H_1 \lambda^{-1}$ are equal.
 \end{claim}

 Indeed, let $K \subset [\mathbb{G}(\R),\mathbb{G}(\R)]$ be a compact subset such that $\lambda H_1 \lambda^{-1} \subset H_1 K$. We proceed by induction on the length of the upper central series. If $[\mathbb{G}(\R),\mathbb{G}(\R)] \simeq \R^n$ for some $n \in \N$, the result is obvious. Otherwise, let $Z$ be the centre of $[\mathbb{G}(\R),\mathbb{G}(\R)]$, by induction hypothesis the projections of $H_1$ and $\lambda H_1\lambda^{-1}$ to $[\mathbb{G}(\R),\mathbb{G}(\R)]/Z$ are equal. So choose $g \in \lambda H_1\lambda^{-1} \setminus H_1$, there is $z \in Z$ such that $gz \in H_1$, moreover for all $n \in \N$ there are $h_n \in H_1$ and $k_n \in K$ such that $g^n=h_nk_n$. As a consequence, 
 $$ \forall n \in \N,\ z^nk_n =z^n h_n^{-1}g^n = h_n^{-1}(gz)^n \in H_1. $$
 Thus, 
 $$ \forall n \in \N, \log(z) + \frac{\log(k_n)}{n} \in \log(H_1), $$
 where $\log$ is the logarithm map from $[\mathbb{G}(\R),\mathbb{G}(\R)]$ to its Lie alebra. But $\log(H_1)$ is closed (since $\log$ is a homeomorphism) so we get $z \in H_1$ and $g \in H_1$, hence Claim \ref{Claim normality}.
 
 \medbreak

 Now, $\L$ is Zariski dense and $H_1$ is connected, so $H_1$ is normal. Moreover, as $H_1$ is connected in a unipotent subgroup, $H_1$ is the group of $\R$-points of an algebraic subgroup $\mathbb{H}_1$ normal in $\mathbb{G}$. 
 
 Therefore, the natural map $\mathbb{G}(\R)/\mathbb{H}_1(\R) \rightarrow (\mathbb{G}/\mathbb{H}_1)(\R)$ is an embedding and its image contains the connected component of the identity in $(\mathbb{G}/\mathbb{H}_1)(\R)$. Moreover, the Zariski-closure of the image of $\L^{\infty}$ contains the image of $\mathbb{G}(\R)$. 
 
 Set $p: \mathbb{G}(\R)\rightarrow (\mathbb{G}/\mathbb{H}_1)(\R)$ the canonical projection, $\tilde{\Lambda}:=p(\Lambda)$ and $\tilde{G}$ the Zariski-closure of $\tilde{\L}$.
 
 \begin{claim}\label{small commutator close to centraliser}
  The approximate subgroup $\tilde{\L}$ is compactly commensurable to a closed connected subgroup of the centre $Z(\tilde{G})$ of $\tilde{G}$. 
 \end{claim}
 
 Let us first show how Proposition \ref{Approximate subgroups in soluble algebraic Groups} follows from this claim. 
 
 There is some closed connected subgroup $V \subset Z(\tilde{G})$ such that $\tilde{\L}$ is compactly commensurable to $V$. Hence, we can find a compact subset $K_1 \subset \mathbb{G}(\R)$ such that 
 $$ \tilde{\L} \subset p(K_1)V \text{ and } V \subset p(K_1)\tilde{\L}. $$
 In particular, according to the first inclusion, $\L \subset K_1p^{-1}(V)$. On the other hand, there is $K_2 \subset \mathbb{G}(\R)$ compact such that $H_1 \subset K_2\L $, where $H_1$ is the subgroup defined above. Finally, 
 $$ p^{-1}(V) \subset K_1\L H_1 = K_1 H_1\L  \subset  K_1K_2 \L^2. $$
 So $p^{-1}(V)$ is the subgroup we are looking for. 
 
 \bigbreak 
 \bigbreak
 
 Now, let us move to the proof of Claim \ref{small commutator close to centraliser}. We will use the fact that the set of commutators of elements of $\tilde{\L}$ is relatively compact to show that $\tilde{\L}$ is contained in a `neighbourhood' of the centre. 
 
 The group $\tilde{\Lambda}^{\infty} \cap [\tilde{G},\tilde{G}]$ is co-compact in $[\tilde{G},\tilde{G}]$. So there is $k \in \N$ such that 
 $$\Span_{\R}(\log(\tilde{\Lambda}^k \cap [\tilde{G},\tilde{G}]))= \mathfrak{Lie}([\tilde{G},\tilde{G}]),$$
 where $\log$ is the logarithm map from $[\tilde{G},\tilde{G}]$ to its Lie algebra.
 In addition, 
 $$\bigcup\limits_{\lambda \in \tilde{\Lambda}} \lambda \left( \tilde{\Lambda}^k \cap [\tilde{G},\tilde{G}] \right) \lambda^{-1} \subset \tilde{\Lambda}^{k+2} \cap [\tilde{G},\tilde{G}].$$
 But the right-hand side is a relatively compact set. Hence, $(\rest{\Ad(\lambda)}{\mathfrak{Lie}([\tilde{G},\tilde{G}])})_{\lambda \in \L}$ is a uniformly bounded family of linear operators. Since
 \begin{align*}
 \rho: (\mathbb{G}/\mathbb{H})(\R) & \rightarrow \GL(\mathfrak{Lie}([\tilde{G},\tilde{G}])) \\
  g & \mapsto \rest{\Ad(g)}{\mathfrak{Lie}([\tilde{G},\tilde{G}])}
 \end{align*}
is an algebraic group homomorphism, there is a compact set $K \subset \tilde{G}$ such that  
 $$\tilde{\Lambda} \subset \ker(\rho)K.$$ The Kernel $Z:=\ker(\rho)$ has finitely many connected components so there is a compact set $K_3$ such that $\tilde{\Lambda} \subset Z^0K_3$ where $Z^0$ is the connected component of the identity. 
 
 Now, for any $g \in \tilde{G}$ define the map 
 \begin{align*}
  \theta_g : & Z  \rightarrow \tilde{G} \\
	     & h \mapsto [g,h],
 \end{align*}
 where $[g,h]$ denotes $ghg^{-1}h^{-1}$.
 For $h_1,h_2 \in Z$ we have 
 \begin{align*}
  \theta_g(h_1)\theta_g(h_2) & = gh_1g^{-1}h_1^{-1}gh_2g^{-1}h_2^{-1}\\
			     & = gh_1g^{-1}gh_2g^{-1}h_2^{-1}h_1^{-1} & \text{ as } h_1 \in Z \\
			     & = g(h_1h_2)g^{-1}(h_1h_2)^{-1} \\
			     & =  \theta_g(h_1h_2).
 \end{align*}
 So $\theta_g$ is an algebraic group homomorphism.
 
 For all $\lambda \in \tilde{\Lambda}$ let $f(\lambda)$ denote an element of $Z$ such that $\lambda f(\lambda)^{-1}\in K_3$. Now, for $\gamma\in \tilde{\Lambda}$ and $\lambda \in \tilde{\Lambda}$ we have  
 $$ \gamma \lambda \gamma^{-1}  \lambda^{-1} \in \gamma K_3 \gamma^{-1} \theta_{\gamma}(f(\lambda)) K_3^{-1}. $$
 Thus, 
 $$ \theta_{\gamma}(f(\tilde{\Lambda})) \subset \gamma K_3^{-1} \gamma^{-1} (\tilde{\Lambda}^4 \cap [\tilde{G},\tilde{G}]) K_3. $$ 
 So $\theta_{\gamma}(f(\Lambda))$ is a relatively compact set. 
 
 Now, for any finite family $\mathcal{F}:=\left\{\gamma_1,\ldots,\gamma_n\right\}$ of elements of $\tilde{\Lambda}$, set 
 
 \begin{align*}
  \theta_{\mathcal{F}}: Z & \rightarrow (Z)^n \\
  g & \mapsto (\theta_{\gamma_1}(g), \ldots, \theta_{\gamma_n}(g))
 \end{align*}

 We readily see that
 $$\ker(\theta_{\mathcal{F}}) = \bigcap\limits_{1\leq i \leq n} Z_{\tilde{G}}(\gamma_i) \cap Z. $$
 
 We know that $Z$ is an algebraic subgroup and $\theta_{\mathcal{F}}$ is an algebraic group morphism. Moreover, since $\theta_{\mathcal{F}}(f(\tilde{\Lambda}))$ is relatively compact as a subset of $\tilde{G}^n$ and $\theta_{\mathcal{F}}(Z)$ is closed, it is relatively compact as a subset of $\theta_{\mathcal{F}}(Z)$. Thus, there is a compact set $K_4$ such that 
 $f(\L) \subset K_4 \ker(\theta_{\mathcal{F}})$. Set $K_{\mathcal{F}}:=K_3K_4$, we get $\tilde{\Lambda} \subset K_{\mathcal{F}}\ker(\theta_{\mathcal{F}})$. 
 
 Moreover, $\tilde{\Lambda}^{\infty}$ is Zariski-dense so we can choose $\mathcal{F}$ such that it generates a Zariski dense subgroup. Hence, 
 $$\tilde{\Lambda} \subset K_{\mathcal{F}}\ker(\theta_{\mathcal{F}}) = K_{\mathcal{F}}Z(\tilde{G}),$$ 
 where $Z(\tilde{G})$ is the centre of $\tilde{G}$. But $Z(\tilde{G})$ has a finite number of connected components and the connected component of the identity is isomorphic to $\R^k\times \mathbb{T}^l$ for some $k,l \in \N$. Therefore, there is a central subgroup $W \subset \tilde{G}$ and a compact subset $K_5 \subset \tilde{G}$ such that $\tilde{\L} \subset K_5W$ and $W \simeq \R^k$ 
 
 Finally, choose  a function $g: \tilde{\Lambda} \rightarrow W$ such that for all $\lambda \in \tilde{\L}, g(\lambda^{-1})=g(\lambda)^{-1}$ and  $b(\lambda):=g(\lambda)\lambda^{-1} \in K_5$. There is a finite subset $F \subset G$ such that for all $ \lambda_1,\lambda_2 \in\tilde{\Lambda} $ there is $\lambda \in \tilde{\Lambda}$ satisfying $\lambda\lambda_1^{-1}\lambda_2^{-1} \in F$. As a consequence, 
 $$g(\lambda)g(\lambda_1)^{-1}g(\lambda_2)^{-2}=\lambda\lambda_1^{-1}\lambda_2^{-1}b(\lambda)b(\lambda_1)^{-1}b(\lambda_2)^{-1}  \in FK_5K_5^{-2}.$$
 By Lemma \ref{lemma approximate subgroups in vector spaces} we obtain a closed connected subgroup $V_1$ compactly commensurable to $\tilde{\L}$. As $V_1 \subset Z(\tilde{G})$ it is compactly commensurable to its Zariski-closure $V_2$. The connected component of the identity of $V_2$ is the subgroup we are looking for. 
 \end{proof}

\begin{proof}[Proof of Theorem \ref{Main theorem}.]
  Let $G$ be the Zariski-closure of $\L^{\infty}$. Then $G$ is the group of $\R$-points of a soluble algebraic group. Let $\tilde{G}$ denote the group of $\R$-points of its Zariski-connected component of the identity. Proposition \ref{Approximate subgroups in soluble algebraic Groups} applied to $\L^2 \cap \tilde{G}$ yields Theorem \ref{Main theorem}.
\end{proof}

\section{Consequences of Theorem \ref{Main theorem}}

In \cite{MR689763} Fried and Goldman proved that every soluble subgroup $H$ of $\GL_n(\R)$ admit a \emph{syndetic hull} i.e. a closed connected subgroup of $\GL_n(\R)$ containing $H$ and such that $H$ is co-compact in it. We show how this theorem is a consequence of Theorem \ref{Main theorem}. 

\begin{proposition}[Theorem 1.6,\cite{MR689763}]\label{proposition Fired et Goldman}
 Let $G$ be the $\R$-points of a soluble real algebraic group and $H$ a subgroup. Then there is $B < G$ such that $B$ is a closed connected subgroup (for the Euclidean topology), $H \cap B$ has finite index in $H$ and $H$ and $B$ are compactly commensurable.  
\end{proposition}

\begin{proof}
 Without loss of generality we can assume that $H$ is Zariski-dense. Then applying Proposition \ref{Approximate subgroups in soluble algebraic Groups} to $H$ we get a closed connected normal subgroup $B \triangleleft G$ such that $H$ is compactly commensurable to $B$. So the image of $H$ in $G/B$ via $p:G \rightarrow G/H$ is contained in a compact subgroup $K$. Let $K_0$ be its connected component of the identity in the Euclidean topology, then set $\tilde{H} = H \cap p^{-1}(K_0)$. The subgroup $\tilde{H}$ has finite index in $H$ and is co-compact in $p^{-1}(K_0)$. 
\end{proof}

We also get a generalisation of the well-known fact that closed soluble subgroups of $\GL_n(\R)$ are compactly generated. 

\begin{proposition}\label{compact generation}
Let $G$ be the $\R$-points of a soluble real algebraic group and $\L \subset G$ an approximate subgroup. Then there is a compact subset $K$ such that $\L^2\cap K$ generates $\L^{\infty}$.  
\end{proposition}

\begin{proof}
As a consequence of Proposition \ref{Approximate subgroups in soluble algebraic Groups}, there is a connected subgroup $H \leq G$ such that $G$ and $\L$ are compactly commensurable. Let $K \subset G$ be a compact symmetric subset such that $\L \subset K H$ and $H \subset K \L$. Choose also $V$ a compact neighbourhood of the identity in $H$. As $H$ is connected for the Euclidean topology, $V$ generates $H$. Now, for any $\lambda \in \L$ choose $h \in H$ such that $\lambda h^{-1} \in K$. Since $V$ generates $H$, we can find a sequence $(h_i)_{0\leq i \leq r}$ of elements of $H$ such that $h_0=e,h_r=h$ and $h_{i+1}h_i^{-1} \in V$. In addition, we can find a sequence $(\lambda_i)_{0 \leq i \leq n}$ of elements of $\L$ such that $\lambda_0=e,\lambda_r=\lambda$ and for all $0 \leq i \leq r$, $\lambda_ih_i^{-1} \in K^{-1}=K$. Thus, $\lambda_{i+1}\lambda_i^{-1} \in KV^{-1}K^{-1}$. Finally, $\L^{\infty}$ is generated by $\L^2 \cap KV^{-1}K^{-1}$.
\end{proof}

In particular, when $\L$ is discrete, this implies that $\L^{\infty}$ is finitely generated. This fact will be used in the proof of Theorem \ref{Meyer's theorem soluble} below. 

Finally, we generalise a theorem from \cite[Theorem 4.25]{bjorklund2016approximate}, who handled the nilpotent case. This result is concerned with \emph{strong approximate lattices}. Strong approximate lattices are defined by measure-theoretic conditions on an associated dynamical system called the invariant hull. We refer the reader to \cite[Section 4]{bjorklund2016approximate} for precise definitions. 

\begin{theorem}\label{strong approximate lattices in soluble groups}
 Let $\Lambda \subset G$ be a strong approximate lattice in the group of $\R$-points of a soluble real algebraic group. Then $\Lambda$ is relatively dense.
\end{theorem}

\begin{proof}
 Indeed, according to \cite[Theorem 4.18]{bjorklund2016approximate} any strong approximate lattice is bi-syndetic i.e. there is $K_1 \subset G$ compact such that $G = K_1 \Lambda K_1$. Moreover, $\Lambda$ is Zariski-dense according to \cite{bjorklund2019borel}. Now let $H$  and $K_2$ be given by Proposition \ref{Approximate subgroups in soluble algebraic Groups} so that 
 $$ \L \subset K_2 H \text{ and } H \subset K_2 \L. $$
 Since $H$ is normal we have that $G = K_1K_2HK_1 = K_1K_2K_1H$.  Then $G = K_1K_2K_1K_2 \Lambda$.
\end{proof} 

\section{Uniform approximate lattices in abelian groups}

We will investigate morphisms commensurating approximate subgroups in $\R^n$. This will turn out to be useful in the proof of Theorem \ref{Meyer's theorem soluble}. Our goal is to understand morphisms that commensurate a uniform approximate lattice. Let us start with a result concerning lattices.

\begin{proposition}\label{Approximate sugroup commensurating a lattice}
 Let $\L \subset \GL_n(\R)$ be an approximate subgroup and suppose there are $\Gamma_1 \subset \Gamma_2$ lattices in $\R^n$ such that $\lambda(\Gamma_1) \subset \Gamma_2$ for all $\lambda \in \L$. Then there is $\Xi \subset \L^4$ commensurable to $\L$ such that $\Xi \subset \Aut(\Gamma_1)$. 
\end{proposition}

\begin{proof}
 We can assume that $\Gamma_1 = \Z^n$. Let $m$ be the order of $\Gamma_2/\Gamma_1$ and $p_1,\ldots,p_r$ the prime factors of $m$. Then any matrix in $\L$ has entries lying in $\frac{1}{m}\Z$. 
 
 Set
 \begin{align*}
  \phi: \GL_n(\R) & \rightarrow \R_+^* \\
   M & \mapsto |\det(M)|
  \end{align*}
 then $\phi$ is a group homomorphism and $\phi(\L) \subset \frac{1}{m^n}\Z$  is a discrete approximate subgroup bounded away from $0$ so $\phi(\L)$ is finite. As a consequence, 
 $$\tilde{\L} = \phi^{-1}(\{1\}) \cap \L^2$$
 is an approximate subgroup commensurable to $\L$.
 \medbreak
 Consider the diagonal embedding
 $$ \iota: \SL_n\left(\Z\left[\frac{1}{m}\right]\right) \hookrightarrow \prod\limits_{i=1}^{r} \SL_n(\Q_{p_i}). $$
 Now, $\iota^{-1}(\prod\limits_{i=1}^{r} \SL_n(\Z_{p_i})) = \SL_n(\Z)$, $\iota(\tilde{\L})$ is relatively compact and $\prod\limits_{i=1}^r \SL_n(\Z_{p_i})$ is open. Therefore, there are $\lambda_1,\ldots,\lambda_s \in \tilde{\L}$ such that 
 \begin{align*}
   \iota(\tilde{\L}) & \subset \bigcup\limits_{i=1}^s \iota(\lambda_s) (\prod\limits_{i=1}^r \SL_n(\Z_{p_i})). 
 \end{align*}
 Now, $\tilde{\L}$ is commensurable to the approximate subgroup $$\tilde{\L}^2 \cap \iota^{-1}(\prod\limits_{i=1}^r \SL_n(\Z_{p_i}))= \tilde{\L}^2 \cap \SL_n(\Z).$$
\end{proof}

Now, we can deduce

\begin{proposition}\label{Commensurator quasi-crystals}
  Let $\L \subset \GL_n(\R)$ be an approximate subgroup and suppose there are $\L_1 \subset \L_2$ approximate lattices in $\R^n$ such that $\lambda(\L_1) \subset \L_2$ for all $\lambda \in \L$. Then there are $\Xi \subset \L^4$ commensurable to $\L$ and an injective group homomorphism $\Xi^{\infty} \rightarrow \SL_m(\Z)$ for some $m \geq n$. .   
\end{proposition}

 This result is not needed in the sequel, however it gives a good insight into the remaining part of the proof of Theorem \ref{Meyer's theorem soluble}. Indeed, a similar argument will be used to prove Proposition \ref{Generated group is polycyclic}.

\begin{proof}
 For any $\Xi$ commensurable to $\L_1$ the subgroup $\Xi^{\infty}$ has finite rank. Choose $\Xi $ commensurable to $\L_1$ with minimal rank, then $\Xi^2 \cap \L_1^2$ is a uniform approximate lattice as well and $\rank(\Xi^2 \cap \L_1^2) \leq \rank(\Xi^2)$, so there is equality. Thus, we can assume that $\Xi \subset \L_1^2$.
 
 As a consequence, for all $\lambda \in \L$ the approximate group $\lambda(\Xi)$ is commensurable to $\lambda(\L_1^2)$ which in turn is commensurable to $\L_2^2$. So $\Xi$ and $\lambda(\Xi)$ are commensurable. Hence, $\Xi$ is commensurable to $\Xi^2 \cap \lambda(\Xi^2)$. By minimality of $\rank(\Xi^{\infty})$ we get that
 $$\rank((\Xi^2 \cap \lambda(\Xi^2))^{\infty}) = \rank(\Xi^{\infty})= \rank(\lambda(\Xi^{\infty})), $$
 so $(\Xi^2 \cap \lambda(\Xi^2))^{\infty}<_{f.i.} \Xi^{\infty}.$
 
 Therefore, $\lambda$ is an isomorphism of $\Span_{\Q}(\Xi)$ and as $\Span_{\R}(\Xi)=\R^n$ we get an injective morphism $\L^{\infty} \rightarrow \GL_m(\Q)$ where $m=\rank(\Xi)$. 
 
 Finally, for all $\lambda \in \L$ we have $\lambda(\Xi^{\infty}) \subset \L_2^{\infty} \cap \Span_{\Q}(\Xi)$. Since 
 
$$ \rank\left( \L_2^{\infty} \cap \Span_{\Q}(\Xi)\right)  = \dim_{\Q}(\Span_{\Q}(\Xi)) $$
 we get that $\Xi^{\infty}$ has finite index in $ \L_2^{\infty} \cap \Span_{\Q}(\Xi)$. So Proposition \ref{Approximate sugroup commensurating a lattice} applied to $\Xi^{\infty},  \L_2^{\infty} \cap \Span_{\Q}(\Xi)$ and $\L$ gives the desired morphism. 
\end{proof}

\begin{remark}
 From the proof of Proposition \ref{Commensurator quasi-crystals}, we have that for any discrete approximate lattice $\L \subset \R^n$ the subgroup $\left\{g \in \GL_n(\R) | g(\L) \text{ is commensurable to } \L \right\}$ is isomorphic to a subgroup of $\GL_m(\Q)$ where $m$ is the minimal rank of an approximate subgroup commensurable to $\Xi$. 
\end{remark}

\section{Meyer's Theorem for soluble Lie groups}

We will now turn to the proof of Theorem \ref{Meyer's theorem soluble}. As a first step, let us prove it with an additional assumption. 

\begin{proposition}\label{polycyclic implies Meyer}
 Let $\L \subset G$ be a uniform approximate lattice in a connected soluble Lie group. If $\L^{\infty}$ is polycyclic it is a Meyer subset. 
\end{proposition}

\begin{proof} 
According to a theorem of Auslander (see \cite{10.2307/1970362} or the proof of \cite[Theorem 4.28]{raghunathan1972discrete}), $\L^{\infty}$ admits an embedding as a Zariski-dense lattice in $R$ the group of $\R$-points of a soluble algebraic group. In the following we will consider $\L^{\infty}$ as a subgroup of $R$. Moreover, we can assume without loss of generality that $R$ is Zariski-connected. Indeed, there is a finite index subgroup $\Gamma$ of $\L^{\infty}$ such that the Zariski closure of $\Gamma$ is Zariski-connected. Furthermore, the approximate subgroup $\L^2 \cap \Gamma$ is commensurable to $\L$ according to Lemma \ref{Intersection commensurable approximate subgroups}.

Now, according to Proposition \ref{Approximate subgroups in soluble algebraic Groups} there is a closed connected normal subgroup $N \triangleleft R$ such that $\L$ is compactly commensurable to $N$. Let $p: R \rightarrow R/N$ denote the natural projection. We know that $p(\L)$ is relatively compact, so we can choose a compact neighbourhood $W_0$ of $p(\L)$. Now, $\L$ is compactly commensurable to $\L^{\infty} \cap p^{-1}(W_0)$, so there is a compact subset $K \subset R$ such that 
$$ \L \subset K (\L^{\infty} \cap p^{-1}(W_0)) \text{ and } \L^{\infty} \cap p^{-1}(W_0) \subset K \L. $$
Therefore, 
$$ \L \subset (K\cap \L^{\infty}) (\L^{\infty} \cap p^{-1}(W_0)) \text{ and } \L^{\infty} \cap p^{-1}(W_0) \subset (K \cap \L^{\infty}) \L. $$
But $\L^{\infty}$ is a discrete subgroup in $R$ so $K\cap \L^{\infty}$ is finite and $\L$ is commensurable to $\L^{\infty} \cap p^{-1}(W_0)$. 

Finally, $\rest{p}{(\L^{\infty} \cap p^{-1}(W_0))^{\infty}}$ is a good model for $\L^{\infty} \cap p^{-1}(W_0)$ . Hence, $\L$ is a Meyer subset.  
\end{proof}

\begin{proposition}\label{Generated group is polycyclic}
 Let $\L \subset G$ be a uniform approximate lattice in a connected soluble Lie group. Then there is $\L'$ commensurable to $\L$ such that $(\L')^{\infty}$ is polycyclic. 
\end{proposition}

\begin{proof}
 Let us first show that we can assume $G$ to be simply connected. Indeed, if $G$ is not simply connected we proceed as follows. Let $p:\tilde{G} \rightarrow G$ be a universal cover, then $p^{-1}(\L)$ is a uniform approximate lattice in $\tilde{G}$. Suppose $p^{-1}(\L)$ is commensurable to an approximate subgroup $\L'$ such that $\L'$ generates a polycyclic group. Then $p(\L')$ is commensurable to $\L$ and $p(\L')$ generates a polycyclic group as well.
 
 From now on $G$ is supposed simply connected. Let $N$ denote the nilradical of $G$, $k \in \N$ and $\Xi \subset \L^k \cap N$ be an approximate subgroup.

 First of all, let us show that $\Xi^{\infty}$ is finitely generated. Since $G$ is simply connected, $G$ does not contain any non-trivial compact subgroup. So $N$ does not contain any non-trivial compact subgroup, and thus $N$ is simply connected. Now, $N$ is a connected simply connected nilpotent Lie group so it is the group of $\R$-points of a unipotent algebraic group (see \cite[Theorem 4.1]{raghunathan1972discrete}) and $\Xi$ is a discrete approximate subgroup. Hence, $\Xi^{\infty}$ is finitely generated by Proposition \ref{compact generation}. 
 
 The proof will rely on the following lemma that links finitely generated subgroups of connected simply connected nilpotent Lie group to finite dimensional $\Q$ Lie algebras.
 
   \begin{lemma}{\cite[Chapter IV]{raghunathan1972discrete}}\label{lemma Raghunathan}
   Let $\Gamma \subset N$ be a finitely generated group in a connected simply connected nilpotent Lie group. Then $\Gamma$ is torsion-free nilpotent, $\Q\log(\Gamma)$ is a finite dimensional $\Q$ Lie algebra and $\dim_{\Q}(\Q\log(\Gamma)) = \rank(\Gamma)$.  
  \end{lemma}
  Where the rank of $\Gamma$ is the dimension of its Malcev completion, i.e. the unique connected simply connected nilpotent Lie group that admits a lattice isomorphic to $\Gamma$. Lemma \ref{lemma Raghunathan} is a consequence of \cite[Theorems 2.18, 2.12, 2.10 and 2.11]{raghunathan1972discrete}. 
   
   \medbreak
   
 Now, $\Xi^{\infty}$ is a finitely generated torsion-free nilpotent group so it has finite rank.  Among all approximate subgroups $\Xi$ commensurable to $\L^2 \cap N$ such that there is $k \in \N$ satisfying $\Xi \subset \L^k \cap N$, choose one with minimal rank. Let $\Xi$ denote this approximate subgroup and let $k$ be such that $\Xi \subset \L^k \cap N$. 
 
 Now, for $\lambda \in \L$, $\Xi$ and $\lambda\Xi\lambda^{-1}$ are contained in and commensurable to $\L^{k+2} \cap N$,  so $\Xi^2 \cap \lambda\Xi^2\lambda^{-1}$ is commensurable to $\Xi$. But $\left(\Xi^2 \cap \lambda\Xi^2\lambda^{-1}\right)^{\infty} \subset \Xi^{\infty}$ so they have the same rank. As a consequence, it is a finite index subgroup, so the groups $\Xi^{\infty}$ and $\lambda\Xi^{\infty}\lambda^{-1}$ are commensurable.
 
 So there is $n$ such that for all $\gamma \in \Xi^{\infty}$ , we have $\gamma^n \in \Xi^{\infty} \cap \lambda\Xi^{\infty}\lambda^{-1}$. Therefore, $$n\log(\Xi^{\infty}) \subset \log(\lambda\Xi^{\infty}\lambda^{-1}).$$
 Hence, $$\Q\log(\Xi^{\infty}) = \Q \log(\lambda\Xi^{\infty}\lambda^{-1}),$$
 where $\log$ denotes the logarithm map from $N$ to its Lie algebra.

 Now, $\exp(\Q\log(\Xi^{\infty}))$ is stable under conjugation by elements of $\L^{\infty}$. Moreover, $\exp(\Q\log(\Xi^{\infty}))$ is a group and any finitely generated subgroup in it has rank less than or equal to $\dim_{\Q}(\Q\log(\Xi^{\infty})) = \rank(\Xi^{\infty})$ according to Lemma \ref{lemma Raghunathan}..
  
  Let $\Gamma$ denote the subgroup generated by $\L^{k+2}\cap \exp(\Q\log(\Xi^{\infty}))$. Since $(\L^{k+2} \cap N)^{\infty}$ is finitely generated according to Proposition \ref{compact generation}, $\Gamma$ is finitely generated as well. In addition, it contains $\Xi^{\infty}$. Therefore, $\rank(\Gamma)= \rank(\Xi^{\infty})$ and $\Xi^{\infty}$ has finite index in $\Gamma$. 
  
  Now, there are free abelian groups $\Gamma_1,\Gamma_2 \subset \Q\log(\Xi^{\infty})$ of rank $\dim_{\Q}(\Q\log(\Xi^{\infty}))$ such that 
  $$\Gamma_1 \subset \log(\Xi^{\infty}) \subset \log(\Gamma) \subset \Gamma_2, $$ 
  see \cite[Theorem 2.12]{raghunathan1972discrete}. Moreover, for all $\lambda \in \L$, $\Ad(\lambda)(\Gamma_1) \subset \Gamma_2$ since $\lambda\Xi^{\infty}\lambda^{-1} \subset \Gamma$. According to Proposition \ref{Approximate sugroup commensurating a lattice}, there is $\tilde{\L}\subset \L^4$ commensurable to $\L$ such that $\Ad(\lambda)(\Gamma_1) = \Gamma_1$ for all $\lambda\in \tilde{\L}$. Therefore, the subgroup $H$ of $\Xi^{\infty}$ generated by $\exp(\Gamma_1)$  has finite index in $\Xi^{\infty}$ and $H \cap \tilde{\L}^{\infty}$ is normalised by $\tilde{\L}$. 
  
  Consider $p: \tilde{\L}^{\infty} \rightarrow \tilde{\L}^{\infty}/(H \cap \tilde{\L}^{\infty})$ the canonical projection. Since $\Xi$ is commensurable to $\Xi^2 \cap H$, $\L^{16} \cap N$ is commensurable to $\Xi$ and $\tilde{\L}\subset \L^4$, the commutators of elements of $p(\tilde{\L})$ form a finite set. So $p(\tilde{\L}^2) \cap Z(\tilde{\L}^{\infty}/(H \cap \tilde{\L}^{\infty}))$ is commensurable to $p(\tilde{\L})$. 
  
  Indeed, $p(\tilde{\L}^{\infty})$ is finitely generated so let $\mathcal{F}$ be a generating family $\left\{\gamma_1,\ldots,\gamma_n\right\} \subset p(\tilde{\L})$. Define the map 
  
  \begin{align*}
  \theta_{\mathcal{F}}:  p(\tilde{\L}^{\infty}) &\rightarrow p(\tilde{\L}^{\infty}) \\
			 \gamma & \mapsto ([\gamma_1,\gamma],\ldots,[\gamma_n,\gamma]),
  \end{align*}
  where $[\gamma_i,\gamma]:=\gamma_i\gamma\gamma_i^{-1}\gamma^{-1}$. We can check that for any $\gamma$ in $ \theta_{\mathcal{F}}(p(\tilde{\L}))$ there is $\gamma' \in p(\tilde{\L})$ such that 
  $$ \theta_{\mathcal{F}}^{-1}(\{\gamma\}) = \gamma'Z(\tilde{\L}^{\infty}/(H \cap \tilde{\L}^{\infty})) $$
  where $Z(\tilde{\L}^{\infty}/(H \cap \tilde{\L}^{\infty}))$ denotes the centre of $\tilde{\L}^{\infty}/(H \cap \tilde{\L}^{\infty})$. But $\theta_{\mathcal{F}}(p(\tilde{\L}))$ is finite so there are $\gamma'_1,\ldots,\gamma'_r \in \tilde{\L}^{\infty}/(H \cap \tilde{\L}^{\infty})$ such that $p(\tilde{\L}) \subset \bigcup \gamma'_i Z(\tilde{\L}^{\infty}/(H \cap \tilde{\L}^{\infty}))$. So according to Lemma \ref{Intersection commensurable approximate subgroups} , $p(\tilde{\L})$ is commensurable to $p(\tilde{\L}^2) \cap Z(\tilde{\L}^{\infty}/(H \cap \tilde{\L}^{\infty}))$. Thus, by Lemma \ref{Intersection commensurable approximate subgroups} once again, $\L':=\tilde{\L}^2 \cap p^{-1}(Z(\tilde{\L}^{\infty}/ (H \cap \tilde{\L}^{\infty})))$ is commensurable to $\tilde{\L}$ and $\L$.
  
  Finally, $\L'$ is a uniform approximate lattice in $R$ as it is commensurable to $\L$. Moreover, $ H \cap \tilde{\L}^{\infty} \subset \L'^{\infty}$ is a finitely generated torsion-free nilpotent normal subgroup such that $\L'^{\infty}/(H \cap \tilde{\L}^{\infty})$ is abelian and finitely generated. Hence, $\L'^{\infty}$ is polycyclic. 
 \medbreak

\end{proof}

\begin{proof}[Proof of Theorem \ref{Meyer's theorem soluble}.]
Let $\L \subset G$ be a uniform approximate lattice in a connected soluble Lie group. According to Proposition \ref{Generated group is polycyclic} $\L$ is commensurable to an approximate subgroup $\L'$ with $\L'$ polycyclic. Now, by Proposition \ref{polycyclic implies Meyer} the approximate subgroup $\L'$ is a Meyer set, so $\L$ is a Meyer set as well.
 
\end{proof}

\section{Acknowledgements}

I am deeply grateful to my supervisor, Emmanuel Breuillard, for his patient guidance and encouragements. 

%

\begin{thebibliography}{10}

\bibitem{10.2307/1970362}
Louis Auslander.
\newblock {On a Problem of Philip Hall}.
\newblock {\em Annals of Mathematics}, 86(1):112--116, 1967.

\bibitem{bjorklund2016approximate}
Michael Bj\"{o}rklund and Tobias Hartnick.
\newblock {Approximate lattices}.
\newblock {\em Duke Math. J.}, 167(15):2903--2964, 2018.

\bibitem{bjorklund2019borel}
Michael Bj{\"o}rklund, Tobias Hartnick, and Thierry Stulemeijer.
\newblock {Borel density for approximate lattices}.
\newblock {\em arXiv preprint arXiv:1902.10560}, 2019.

\bibitem{MR3090256}
Emmanuel Breuillard, Ben Green, and Terence Tao.
\newblock {The structure of approximate groups}.
\newblock {\em Publ. Math. Inst. Hautes \'{E}tudes Sci.}, 116:115--221, 2012.

\bibitem{MR3438951}
Pietro~Kreitlon Carolino.
\newblock {\em The {S}tructure of {L}ocally {C}ompact {A}pproximate {G}roups}.
\newblock ProQuest LLC, Ann Arbor, MI, 2015.
\newblock Thesis (Ph.D.)--University of California, Los Angeles.

\bibitem{fish2019extensions}
Alexander Fish.
\newblock {Extensions of {S}chreiber's theorem on discrete approximate
  subgroups in {$\mathbb{R}^d$}}.
\newblock {\em arXiv preprint arXiv:1901.08055}, 2019.

\bibitem{MR689763}
David Fried and William~M. Goldman.
\newblock {Three-dimensional affine crystallographic groups}.
\newblock {\em Adv. in Math.}, 47(1):1--49, 1983.

\bibitem{MR2833482}
Ehud Hrushovski.
\newblock {Stable group theory and approximate subgroups}.
\newblock {\em J. Amer. Math. Soc.}, 25(1):189--243, 2012.

\bibitem{machado2018approximate}
Simon Machado.
\newblock {Approximate lattices and Meyer sets in nilpotent Lie groups}.
\newblock {\em arXiv preprint arXiv:1810.10870}, 2018.

\bibitem{machado2019goodmodels}
Simon Machado.
\newblock {Good models for approximate subgroups}.
\newblock {\em In preparation}, 2019.

\bibitem{MR3345797}
Jean-Cyrille Massicot and Frank~O. Wagner.
\newblock {Approximate subgroups}, {J. \'{E}c. polytech. Math.}
\newblock 2:55--64, 2015.

\bibitem{meyer1972algebraic}
Yves Meyer.
\newblock {\em {Algebraic numbers and harmonic analysis}}, volume~2.
\newblock Elsevier, 1972.

\bibitem{raghunathan1972discrete}
Madabusi~Santanam Raghunathan.
\newblock {Discrete subgroups of Lie groups}.
\newblock {\em Ergebnisse der Mathematik}, 68, 1972.

\bibitem{schreiber1973approximations}
Jean-Pierre Schreiber.
\newblock {Approximations diophantiennes et problemes additifs dans les groupes
  ab{\'e}liens localement compacts}.
\newblock {\em Bull. Soc. Math. France}, 101:297--332, 1973.

\bibitem{springer2010linear}
Tonny~Albert Springer.
\newblock {\em {Linear algebraic groups}}.
\newblock Springer Science \& Business Media, 2010.

\bibitem{MR2501249}
Terence Tao.
\newblock {Product set estimates for non-commutative groups}.
\newblock {\em Combinatorica}, 28(5):547--594, 2008.

\end{thebibliography}
%

\end{document}